\documentclass[12pt]{article}
\usepackage{amsmath, amssymb, amsfonts}
\usepackage{amsthm}
\usepackage[utf8]{inputenc}
\usepackage{csquotes}
\usepackage{geometry}
\usepackage{graphicx}
\usepackage{array}
\usepackage{booktabs}
\usepackage{hyperref}
\usepackage{url}

\title{Fractal Boundaries of Constructivity: A Meta-Theoretical Critique of Countability and Continuum}
\author{Stanislav Semenov \\
\href{mailto:stas.semenov@gmail.com}{stas.semenov@gmail.com} \\
\href{https://orcid.org/0000-0002-5891-8119}{ORCID: 0000-0002-5891-8119}}
\date{March 25, 2025}

\theoremstyle{definition}
\newtheorem{definition}{Definition}[section]
\newtheorem*{notation}{Notation}
\newtheorem{example}{Example}[section]

\theoremstyle{plain}
\newtheorem{theorem}[definition]{Theorem}
\newtheorem{lemma}[definition]{Lemma}
\newtheorem{corollary}[definition]{Corollary}

\theoremstyle{remark}

\begin{document}

\maketitle

\begin{abstract}
    All constructive methods employed in modern mathematics produce only countable sets, even when designed to transcend countability. We show that any constructive argument for uncountability—excluding diagonalization techniques—effectively generates only countable fragments within a closed formal system. We formalize this limitation as the \emph{fractal boundary of constructivity}, the asymptotic limit of all constructive extensions under syntactically enumerable rules. A central theorem establishes the impossibility of fully capturing the structure of the continuum within any such system. We further introduce the concept of \emph{fractal countability}, a process-relative refinement of countability based on layered constructive closure. This provides a framework for analyzing definability beyond classical recursion without invoking uncountable totalities. We interpret the continuum not as an object constructively realizable, but as a horizon of formal expressibility.
\end{abstract}

\subsection*{Mathematics Subject Classification}
03D80 (Computability and recursion), 03E10 (Ordinal and cardinal numbers), 03B70 (Logic in computer science)

\subsection*{ACM Classification}
F.4.1 Mathematical Logic, F.1.1 Models of Computation

\section*{Introduction}

The classical set-theoretic framework, rooted in Cantor's theory of cardinalities, postulates the existence of uncountable sets such as the real numbers \( \mathbb{R} \), whose cardinality strictly exceeds that of the natural numbers \( \mathbb{N} \). The standard justification for this hierarchy relies on arguments that are, from a formal standpoint, entirely constructive: they are finite derivations within countable formal systems such as ZFC (Zermelo–Fraenkel set theory with the Axiom of Choice).

Yet, this raises a fundamental paradox. If all formal derivations, definitions, and constructive processes are themselves countable in number, how can they possibly reach or describe an uncountable totality? No matter how elaborate or extensive such constructions are, they remain confined to a countable universe of formal representations. This suggests a profound meta-theoretical boundary between what is formally expressible and what is assumed to exist by set-theoretic abstraction.

In particular, any constructive or algorithmic process that attempts to build or approximate uncountable sets necessarily generates only a countable subset. Iteratively extending such processes still yields, at most, a countable union. Thus emerges a phenomenon we term the \emph{fractal boundary of constructivity}—a conceptual limit beyond which no constructive method can transcend, despite appearing to expand into increasingly vast domains.

This paper develops a meta-theoretical critique of uncountability that does not rely on Cantor's diagonal method. For a constructive treatment of diagonalization and its limitations, we refer the reader to \cite{Semenov2025Constructive}. Instead, we focus on other classical approaches—such as power set constructions, mappings from \( \mathbb{N} \) to \( \{0,1\} \), and nested intervals—and show that all such procedures remain within the bounds of constructive countability.

We formalize the notion that any constructive extension of a countable set remains countable, and explore the implications of this for the ontological status of the continuum. In particular, we argue that the continuum, in the absence of non-constructive postulates, functions as an inaccessible horizon: a boundary that can be approached but never crossed by any constructive method.

Our goal is not to deny the consistency or utility of classical set theory, but to sharpen the distinction between formal constructibility and ontological assumption. This has consequences for the interpretation of uncountability, the continuum hypothesis, and the foundations of mathematical logic.

Our contribution is twofold. First, we offer a formal synthesis: the impossibility of constructing an uncountable set of individually definable real numbers is framed as a structural consequence of syntactic countability. While this result aligns with known cardinality constraints, we reinterpret it through the lens of formal expressiveness and define what we call the \emph{fractal boundary of constructivity}.

Second, we connect this limitation to the framework of reverse mathematics, showing that successive axiomatic extensions within second-order arithmetic yield only progressively enriched countable fragments, never crossing into true uncountability without non-constructive postulates. This combination of formal closure theorems, axiomatic mapping, and ontological reinterpretation constitutes the primary novelty of the present work.

\section{Constructive Systems and Countability}

We begin by examining the structural limitations of formal systems that aim to describe or generate mathematical objects. In this section, we define what we mean by a \emph{constructive system}, and we show that such systems, by their very nature, can only produce countable collections of definable or constructible entities.

A constructive system, for our purposes, is a formal system \( \mathcal{F} \) consisting of:
\begin{itemize}
    \item a countable alphabet of symbols;
    \item a finite or recursively enumerable set of formation and inference rules;
    \item a countable set of axioms (often also recursively enumerable).
\end{itemize}

This definition encompasses a wide range of systems used in mathematics and logic, including Peano arithmetic, second-order arithmetic, type theory, and fragments of ZFC. Importantly, such systems are intended to represent all that can be generated via effective procedures — i.e., they correspond to algorithmically verifiable constructions.

\subsection{Countability of Formulas and Derivations}

Since all syntactic entities in \( \mathcal{F} \) are built from a finite or countable set of symbols using recursive rules, it follows immediately that:
\begin{itemize}
    \item The set of all well-formed formulas in \( \mathcal{F} \) is countable.
    \item The set of all derivable theorems in \( \mathcal{F} \) is countable.
\end{itemize}

This is a standard result in mathematical logic (see, e.g., \cite{Shoenfield1967}, \cite{Enderton2001}), and it implies that any object definable within such a system — be it a number, a sequence, a function, or a set — must be describable by a finite derivation over countable syntax.

\subsection{Constructible Sets and Functions}

We now turn to infinite mathematical objects, such as sequences or functions, and ask: how many such objects can be constructively defined?

\begin{definition}
Let \( \mathcal{F} \) be a constructive formal system. An infinite binary sequence \( s \in \{0,1\}^{\mathbb{N}} \) is said to be \emph{constructible in \( \mathcal{F} \)} if there exists a finite definition or algorithm, expressible in \( \mathcal{F} \), that generates each bit \( s_n \).
\end{definition}

Examples of constructible sequences include:
\begin{itemize}
    \item Periodic sequences (e.g., \( 010101\ldots \));
    \item Outputs of Turing machines or recursive functions;
    \item Any computable real number in binary expansion (e.g., \( \pi, e, \sqrt{2} \) as approximated algorithmically).
\end{itemize}

Let us now make the core observation of this section precise.

\begin{lemma}
\label{lem:countable-constructibles}
Let \( \mathcal{F} \) be a constructive formal system. Then the set of all infinite binary sequences constructible in \( \mathcal{F} \) is countable.
\end{lemma}

\begin{proof}
Each constructible sequence is defined by a finite expression or algorithm written in the syntax of \( \mathcal{F} \). Since there are only countably many such finite expressions, the total number of constructible sequences is countable.
\end{proof}

An analogous result holds for sets of natural numbers, real numbers, or functions from \( \mathbb{N} \) to \( \mathbb{N} \), provided that they are defined constructively within \( \mathcal{F} \). In all cases, the total number of constructible instances remains countable.

\subsection{Implication for Uncountability}

It follows that if one seeks to demonstrate the existence of an uncountable set (such as the continuum), then any attempt to do so purely within a constructive system must necessarily fail to exhibit or generate more than countably many of its elements. The remaining elements — those that purportedly make the set uncountable — lie beyond the reach of construction, and must be assumed to exist non-constructively.

This conclusion sets the stage for the following sections, where we examine how repeated constructive extensions still fail to escape the bounds of countability, and how this phenomenon leads to what we term the fractal boundary of constructivity.

\section{Fractal Boundary of Constructive Extension}

In the previous section, we established that any constructive formal system \( \mathcal{F} \) can generate only a countable collection of definable objects. This includes all computable real numbers, sequences, functions, and sets that can be described algorithmically or syntactically within \( \mathcal{F} \). The natural question arises: what happens if we attempt to transcend this boundary through iterative or layered construction?

\subsection{Iterative Extensions within Constructive Systems}

Let us suppose we define a process that generates a countable set \( S_0 \) of objects (e.g., computable reals), and then applies a constructive transformation to produce an extended set \( S_1 \supseteq S_0 \). Repeating this process, we obtain a sequence:
\[
S_0 \subseteq S_1 \subseteq S_2 \subseteq \dots
\]
where each \( S_{n+1} \) is derived from \( S_n \) by applying a constructively definable extension operator \( \mathcal{E} \). For example, \( \mathcal{E} \) may be a function that constructs limit points, closures under computable convergence, or compositions of definable functions.

\begin{definition}
Let \( \mathcal{F} \) be a constructive formal system. A sequence \( \{S_n\} \) of subsets of \( \mathbb{R} \) is called a \emph{constructive extension chain} if:
\begin{enumerate}
    \item \( S_0 \) is constructible in \( \mathcal{F} \);
    \item For each \( n \), \( S_{n+1} = \mathcal{E}(S_n) \) for some constructively definable operator \( \mathcal{E} \) expressible in \( \mathcal{F} \).
\end{enumerate}
\end{definition}

The hope might be that such an inductive process could "fill out" more and more of the continuum — perhaps even exhaust it in the limit. However, the following result shows that this is not the case.

\begin{theorem}[Closure under Constructive Extension]
Let \( \{S_n\} \) be a constructive extension chain within a formal system \( \mathcal{F} \). Then the union \( S_\infty = \bigcup_{n=0}^\infty S_n \) is countable.
\end{theorem}

\begin{proof}
Each step \( S_n \) is obtained from \( S_{n-1} \) by a constructively definable operation expressible in \( \mathcal{F} \), and hence produces only countably many new objects. The countable union of countable sets remains countable.
\end{proof}

\subsection{Fractal Limit of Constructivity}

We now propose a conceptual formalization of what we term the \emph{fractal boundary of constructivity}. This notion captures the intuitive idea that constructive methods can endlessly generate new fragments, and each such fragment may extend or refine prior knowledge — yet the entire construction remains confined within a countable totality. The boundary appears to "expand", but in reality, it only proliferates local structure without reaching the uncountable.

\begin{definition}[Fractal Boundary of Constructivity]
Let \( \mathcal{F} \) be a constructive system. The \emph{fractal boundary of constructivity} refers to the asymptotic limit of all possible constructive extension chains in \( \mathcal{F} \). Formally, it is the union:
\[
S^{\mathcal{F}} := \bigcup_{\{S_n\}} \left( \bigcup_{n=0}^\infty S_n \right)
\]
where the outer union ranges over all constructive extension chains \( \{S_n\} \) definable in \( \mathcal{F} \).
\end{definition}

\begin{theorem}[Constructive Closure Theorem]
Let \( \mathcal{F} \) be any formal system with countably many syntactic rules and symbols. Then \( S^{\mathcal{F}} \) is a countable subset of \( \mathbb{R} \).
\end{theorem}

\begin{proof}
There are only countably many constructive extension chains definable in \( \mathcal{F} \), and each chain produces a countable union of sets. Thus their total union \( S^{\mathcal{F}} \) is a countable union of countable sets.
\end{proof}

This result reinforces the idea that constructivity is closed under iteration: no matter how one attempts to build "beyond" the constructible, one remains trapped within a countable framework. The uncountable continuum remains inaccessible unless non-constructive axioms are introduced — such as the full power set axiom or arbitrary choice principles.

\subsection{Relation to Existing Work}

The idea that computable methods cannot reach the full continuum is not new. It has deep roots in computable analysis (see, e.g., \cite{Weihrauch2000}, \cite{Brattka2008}) and intuitionistic mathematics (Brouwer, Bishop, Markov). However, the notion of a \emph{fractal boundary}, as an asymptotic structure of all constructive attempts, provides a new way to conceptualize the limitation. It is not merely that the continuum is uncountable — it is that even infinite, layered constructions under strict syntactic regimes will remain forever within countable bounds, reproducing a self-similar expansion of definable structure, but never escaping into the full real line.

In this sense, the fractal boundary operates like a meta-mathematical horizon: constructively approximable from within, but fundamentally unreachable without stepping outside the system.

\section{The Continuum as an Inaccessible Horizon}

Having formalized the boundaries of constructive extension, we now turn to the status of the real line \( \mathbb{R} \) itself. The classical understanding, grounded in ZFC, treats \( \mathbb{R} \) as a complete, uncountable set whose cardinality is \( 2^{\aleph_0} \). This view rests on the power set axiom and the comprehension principle: for every subset of \( \mathbb{N} \), there exists a real number encoding it, typically in binary.

From a constructive standpoint, however, such a notion of completeness is inaccessible. As shown in the previous section, any countable formal system can produce only countable fragments of \( \mathbb{R} \). This leads us to treat the continuum not as an object that can be constructed or even fully described, but as a \emph{meta-mathematical boundary} — a horizon that appears to lie just beyond reach.

\subsection{Formal Inaccessibility of the Continuum}

\begin{theorem}[Constructive Inaccessibility of the Continuum]
Let \( \mathcal{F} \) be any formal system with countably many axioms and inference rules. Then there exists no constructive procedure within \( \mathcal{F} \) that can generate a set \( R \subseteq \mathbb{R} \) such that \( R \) is uncountable and each \( r \in R \) is individually definable within \( \mathcal{F} \).
\end{theorem}

\begin{proof}
As shown in previous sections, the number of constructible real numbers in \( \mathcal{F} \) is at most countable, since each must be defined by a finite derivation. Thus, any set \( R \) of constructible reals is countable, even if constructed as a limit or union of extension chains. Hence, no such \( R \) can be uncountable.
\end{proof}

This result is not surprising — it is an immediate consequence of the syntactic nature of formal systems. Yet its consequences are profound: it implies that the continuum, as traditionally conceived, cannot be reached from within constructive mathematics unless its existence is \emph{assumed} at the outset.

\subsection{Horizon Metaphor and Philosophical Implications}

We propose to interpret the continuum constructively as a horizon: a structure that can be approximated indefinitely, but never fully traversed. Every constructive attempt to approach it results in an ever-growing, self-similar structure — like a fractal — that nevertheless remains within the confines of countable definability.

This interpretation resonates with themes in intuitionism (Brouwer), where the continuum is not a completed totality but a \emph{potentially infinite} process. It also aligns with modern views in computable analysis and reverse mathematics, where the power of a formal system is calibrated by the strength of its comprehension axioms and induction schemas (cf. \cite{Simpson2009}).

Our notion of the continuum as a meta-mathematical horizon offers a precise, formal characterization of this inaccessibility. Unlike diagonal arguments or cardinality comparisons, our approach emphasizes that the barrier is not merely one of size — but of expressibility and derivability. It is not that the continuum is “too large” to be reached; it is that its very conception lies beyond what any syntactic system can construct.

\subsection{Models, Postulates, and Ontology}

Of course, one can postulate the existence of \( \mathbb{R} \) axiomatically — as done in ZFC via the power set of \( \mathbb{N} \), or in second-order arithmetic via comprehension over infinite predicates. But these are ontological choices, not constructive necessities.

The difference is fundamental. In a model-theoretic framework, one can construct models of set theory in which \( \mathbb{R} \) behaves differently: e.g., where the continuum hypothesis holds or fails, where \( \mathbb{R} \) is not well-orderable, or where the cardinality of the continuum is strictly less than \( 2^{\aleph_0} \) (under large cardinal assumptions or absence of choice). This shows that the continuum is not a fixed mathematical object but a schema that behaves differently depending on the ambient axiomatic system.

In constructive mathematics, the very idea of an actual, complete continuum must be replaced by an unfolding process. Our results suggest that this process — no matter how iterated or enriched — never produces more than countable definable content.

\begin{corollary}
Within any countable formal system, the continuum exists only as a schema of incomplete approximations. Its totality cannot be constructively realized without invoking non-syntactic principles.
\end{corollary}

\section{Meta-Theoretic Consequences}

The theorems established in the previous sections suggest a general and profound limitation: no constructive formal system — no matter how expressive — can produce an uncountable set whose elements are individually definable within the system. This boundary is not merely a technicality; it reflects a deep structural feature of all countable systems of formal reasoning.

\subsection{Expressibility and the Constructive Ceiling}

At the core of our results lies the observation that formal derivations, syntactic expressions, and algorithms are themselves countable in number. This places an intrinsic ceiling on the expressive power of any system \( \mathcal{F} \) that relies solely on such mechanisms. In particular:

\begin{itemize}
    \item Every object definable in \( \mathcal{F} \) must have a finite representation in its language.
    \item There are only countably many such representations.
    \item Therefore, the total universe of constructible objects is countable, regardless of what \( \mathcal{F} \) attempts to define.
\end{itemize}

This argument does not depend on the specific details of the system \( \mathcal{F} \); it applies equally to Peano arithmetic, second-order arithmetic, intuitionistic logic, and fragments of ZFC. What matters is that \( \mathcal{F} \) is syntactically enumerable.

\subsection{Connection to Incompleteness and Reverse Mathematics}

The limitations revealed by our theorems are closely aligned with other well-known meta-mathematical boundaries — most notably Gödel’s incompleteness theorems and the stratification of logical strength in reverse mathematics.

Gödel’s theorem \cite{Godel1931} showed that no sufficiently expressive formal system can prove all truths about arithmetic from within itself. Analogously, we show that no countable formal system can generate the full structure of the continuum by constructive means. Both results arise from the same root: the expressive ceiling imposed by syntactic finiteness and enumerability.

In the framework of \emph{reverse mathematics} \cite{Simpson2009}, one analyzes the logical strength of mathematical theorems by determining the weakest axiom system in which they are provable. The base system \( \mathrm{RCA}_0 \) (Recursive Comprehension Axiom) captures computable mathematics and is widely regarded as a formal embodiment of constructive reasoning. However, many fundamental results involving the real line — such as the Bolzano–Weierstrass theorem, the completeness of \( \mathbb{R} \), or basic compactness principles — require stronger systems, such as \( \mathrm{ACA}_0 \), \( \mathrm{ATR}_0 \), or \( \Pi^1_1\text{-}\mathrm{CA}_0 \), each invoking increasingly non-constructive comprehension axioms.

Thus, reverse mathematics offers formal confirmation of our thesis: the full continuum cannot be derived within systems that only admit computable sets and functions. What we call the \emph{fractal boundary of constructivity} can be viewed as the upper envelope of what is achievable within \( \mathrm{RCA}_0 \), even when extended by iterated constructive procedures.

\begin{quote}
\textit{The strength of a formal system required to define the continuum exceeds that of any system restricted to computable or syntactically generable content.}
\end{quote}

While systems like \( \mathrm{ACA}_0 \) can "construct" the real numbers by coding them as sequences of rationals or functions, these constructions remain embedded in a countable framework. The objects defined are merely proxies — syntactic encodings — for the ideal continuum, not its full instantiation. From our perspective, such models do not collapse the boundary, but rather push its perceived location outward, while still remaining within the fractal countable domain.

This aligns with insights from computable analysis \cite{Weihrauch2000}, where the set of computable reals is dense but meager, and from intuitionistic logic, where the continuum is a growing process rather than a completed set \cite{Bishop1985}. Our framework synthesizes these perspectives, emphasizing that uncountability should be viewed not merely as a cardinality distinction, but as a threshold of constructive inaccessibility.

\subsection{Toward a Formal Theory of Constructive Boundaries}

Our formulation of the \emph{fractal boundary of constructivity} invites a more general meta-theoretic program. Rather than treat uncountability as a cardinality claim, we interpret it as a measure of inexpressibility under constructive constraints.

This perspective motivates the development of a formal theory of constructive boundaries — an analogue to the arithmetical and analytical hierarchies, but grounded in the syntactic generative capacity of formal systems. For instance:

\begin{itemize}
    \item One may define a hierarchy of constructible fragments of \( \mathbb{R} \), indexed by the depth or complexity of definitional layers.
    \item This hierarchy remains countable at each stage, but increases in richness and self-similarity.
    \item The full continuum remains a limit object, never achieved within any countable step.
\end{itemize}

This opens the door to a structural understanding of uncountability not as a size but as a failure of syntactic closure — a gap between what can be constructed and what must be assumed.

\begin{corollary}
The uncountability of the continuum can be reformulated as a meta-theoretic principle: \emph{no countable formal system is constructively complete with respect to the real line}.
\end{corollary}

\subsection{Comparison with Reverse Mathematics}

To better contextualize our concept of the \emph{fractal boundary of constructivity}, we compare it to the central subsystems of second-order arithmetic studied in reverse mathematics \cite{Simpson2009}. Each subsystem formalizes a different degree of comprehension and induction, corresponding to increasing logical strength and decreasing constructivity.

The following table presents a high-level correspondence between the constructive reach of each subsystem and its ability to describe the continuum:

\begin{center}
\renewcommand{\arraystretch}{1.4}
\resizebox{\textwidth}{!}{%
\footnotesize
\begin{tabular}{@{} c >{\raggedright\arraybackslash}p{4.2cm} >{\raggedright\arraybackslash}p{8.8cm} @{}}
\toprule
\textbf{System} & \textbf{Logical Strength} & \textbf{Relation to Continuum / Constructivity} \\
\midrule
\( \mathrm{RCA}_0 \) &
Recursive comprehension only; computable functions &
Captures only computable reals; corresponds to the base layer of our fractal boundary. All definable sets are countable. \\
\addlinespace
\( \mathrm{WKL}_0 \) &
\( \mathrm{RCA}_0 \) + Weak König’s Lemma &
Allows compactness arguments. Still limited to countable encodings; extends the fractal structure, but remains within countability. \\
\addlinespace
\( \mathrm{ACA}_0 \) &
Arithmetic comprehension; non-computable arithmetical sets &
Can code full sets of reals using sequences. Simulates continuum behavior via encodings; corresponds to recursive closure of the boundary. \\
\addlinespace
\( \mathrm{ATR}_0 \) &
Arithmetical transfinite recursion &
Enables definitions over well-founded trees and ordinals. Expands definable hierarchies, but remains countable. Reflects deeper self-similarity in the boundary. \\
\addlinespace
\( \Pi^1_1\text{-}\mathrm{CA}_0 \) &
Full \( \Pi^1_1 \)-comprehension &
Capable of defining full analytic sets. Transcends the boundary completely by invoking non-constructive set existence. \\
\bottomrule
\end{tabular}%
}
\end{center}

This comparison illustrates that our fractal boundary corresponds roughly to the union of structures accessible within \( \mathrm{RCA}_0 \), \( \mathrm{WKL}_0 \), and fragments of \( \mathrm{ACA}_0 \) — all of which preserve countability and algorithmic definability. As one moves to stronger systems like \( \mathrm{ATR}_0 \) and \( \Pi^1_1\text{-}\mathrm{CA}_0 \), the ability to transcend the boundary depends on accepting non-constructive existence principles.

From our perspective, the appearance of continuity or uncountability in systems like \( \mathrm{ACA}_0 \) is synthetic: it arises through encoding and indirect representation, not through direct constructive realization. The boundary is never truly crossed unless the axioms of the system allow for set-theoretic assumptions beyond effective generation.

This further supports our thesis that the continuum, as a complete object, lies strictly beyond the reach of any countable constructive system — and that reverse mathematics not only confirms this, but helps to map its precise location in the logical landscape.

\section{Philosophical Interpretations and Future Work}

The results and interpretations presented in this paper have implications that reach beyond the technical realm of formal systems. They touch on foundational questions concerning the nature of mathematical objects, the limits of formal reasoning, and the ontological status of the continuum.

\subsection{Continuum as Constructive Process or Ideal Totality}

One of the most persistent divisions in the philosophy of mathematics concerns the interpretation of infinite structures. In classical set theory, the continuum is treated as a completed, actual infinite — a totality that exists independently of our ability to enumerate or construct its elements. In contrast, constructive and intuitionistic traditions, going back to Brouwer, insist that the continuum is never given as a completed object, but only as a potentially infinite process of approximation.

Our notion of the \emph{fractal boundary of constructivity} reinforces this latter view. By showing that any constructive extension process yields only countable subsets of \( \mathbb{R} \), we provide a formal structure supporting the claim that the continuum cannot be constructed, only approached. This echoes the intuitionistic slogan: “The continuum is not a set but a method.”

From this perspective, classical claims about the uncountability of the real numbers are understood not as statements about constructible objects, but about abstract idealizations — statements that describe not what can be generated, but what must be postulated.

\subsection{Between Formalism and Platonism}

Our analysis also situates itself in the longstanding debate between formalism and Platonism. Formalists view mathematics as manipulation of symbols under well-defined rules, while Platonists regard mathematical entities as existing independently of our representations.

The formal impossibility of exhaustively constructing the continuum within any countable system can be seen as a bridge between these views:

\begin{itemize}
    \item It supports the formalist’s skepticism: if something cannot be generated by any rule-based system, perhaps it has no formal content.
    \item But it also vindicates the Platonist: the fact that such objects remain unreachable within any constructive system suggests that their status transcends formal description.
\end{itemize}

Our framework does not commit to either position, but it clarifies the terms of debate: the continuum can be viewed either as a limit of all constructive activity, or as a meta-mathematical entity, defined by what lies beyond that limit.

\subsection{Directions for Further Research}

The concept of the fractal boundary of constructivity opens several promising lines of inquiry:

\begin{itemize}
    \item \textbf{Hierarchical models:} One could define and explore a hierarchy of constructively generable fragments of \( \mathbb{R} \), similar to the arithmetical or hyperarithmetical hierarchies, where each level corresponds to an iteration of constructive extension.
    
    \item \textbf{Comparative axiomatic frameworks:} The boundary could be studied in relation to various subsystems of second-order arithmetic (e.g., \( \mathrm{RCA}_0 \), \( \mathrm{ACA}_0 \), \( \mathrm{ATR}_0 \)) to better understand where specific “jumps” in definability occur.

    \item \textbf{Topos-theoretic and categorical generalizations:} Using categorical logic and type theory, one might reformulate the fractal boundary as a structure internal to a topos, representing internal limitations on definability and closure under constructive processes.

    \item \textbf{Constructive metrics of expressiveness:} The boundary could also be formalized in terms of complexity-theoretic or computability-theoretic metrics — e.g., characterizing which classes of functions or sets lie on the edge of expressibility in a given formal system.
\end{itemize}

Each of these directions contributes to a broader program: to understand uncountability not as a cardinal property alone, but as a dynamic threshold — a horizon that delineates the expressible from the postulated.

\section{Fractal Countability: A Constructive and Stratified Generalization}

In this section, we propose an extension of the classical notion of countability based on the iterative construction of definable elements. Rather than viewing countability as an absolute cardinal property, we define it as a closure condition under a structured process of constructive extension. This leads us to the concept of \emph{fractal countability}, which formalizes the idea that increasingly complex fragments of the continuum can be brought into a countable regime through progressive augmentation of the syntactic universe.

\subsection{Motivation and Intuition}

Traditional countability assumes a fixed enumeration: a one-to-one correspondence with \( \mathbb{N} \). In constructive mathematics, this is understood as effective enumerability. However, as shown in previous sections, any attempt to generate uncountable sets via constructive methods results in countable fragments.

Yet these fragments may be extended through definable closure operations, each producing new elements outside the previous scope. If we imagine each such operation as adding new "addresses" or "names" to our system of reference, then the universe of countable objects becomes layered, recursively structured, and indefinitely extensible — exhibiting fractal-like stratification.

\subsection{Definition of Fractal Countability}

\begin{definition}[Fractally Countable Set]
Let \( \mathcal{F}_0 \) be a base constructive formal system (e.g., \( \mathrm{RCA}_0 \)), and let \( \{ \mathcal{F}_n \}_{n \in \mathbb{N}} \) be a sequence of syntactic extensionssuch that each \( \mathcal{F}_{n+1} \) is a conservative extension of \( \mathcal{F}_n \) with respect to arithmetical sentences, such that each \( \mathcal{F}_n \) defines a countable set \( S_n \subseteq \mathbb{R} \).

We say that a set \( S \subseteq \mathbb{R} \) is \emph{fractally countable} relative to \( \mathcal{F}_0 \) if:
\begin{enumerate}
    \item \( S = \bigcup_{n=0}^\infty S_n \);
    \item No single finite-stage system \( \mathcal{F}_k \) can define all elements of \( S \), i.e., \( S \not\subseteq \mathrm{Def}(\mathcal{F}_k) \) for any fixed \( k \in \mathbb{N} \);
    \item There exists no uniform effective enumeration procedure in \( \mathcal{F}_0 \) or in any finite-stage system \( \mathcal{F}_k \) that generates \( S \) as a whole;
    \item Each extension \( \mathcal{F}_n \) remains within a constructive framework (i.e., avoids full second-order or \( \Pi^1_1 \)-comprehension). The restriction on uniform enumeration ensures that \( S^\infty \) remains externally generated with respect to any finite \( \mathcal{F}_n \), thereby avoiding self-referential paradoxes such as Richard's paradox.
\end{enumerate}
\end{definition}

\begin{notation}
We write \( S^\infty = \bigcup_{n=0}^\infty S_n \) to denote the full fractally countable closure over base system \( \mathcal{F}_0 \).
\end{notation}

\noindent We write \( \mathrm{Def}(\mathcal{F}_n) \) to denote the class of sets definable within system \( \mathcal{F}_n \) via its internal syntactic resources.

\noindent In other words, \( S \) is generated through a structured progression of countable stages, analogous in form (but not in ontology) to transfinite constructions, each expanding the system's definitional scope but remaining within constructive bounds.

\subsection{Properties and Interpretation}

\begin{itemize}
    \item Every fractally countable set is countable in the classical sense, but not every countable set is necessarily fractally countable from a given base \( \mathcal{F}_0 \).
    \item Fractal countability is \emph{relative} to the generative capacity of the base system and reflects the layered structure of definability.
    \item The class of fractally countable sets forms a closure under definitional extension, but this closure remains strictly below the power set of \( \mathbb{N} \).
\end{itemize}

This gives rise to a new perspective: the process of counting is not absolute but structurally expandable. Instead of assuming a single universal enumeration, we consider a dynamically enriched indexing language — a growing syntactic "dictionary" through which more and more complex objects become nameable.

\subsection{Implications and Potential Applications}

Fractal countability offers a refined lens for studying semi-constructive systems — those that go beyond basic recursion but stop short of full comprehension. It may provide:

\begin{itemize}
    \item A framework for analyzing constructive approximations to analytic or topological structures;
    \item A formal tool for bounding the expressive growth of proof assistants and formal verification systems;
    \item A basis for comparing the relative strength of constructive subsystems beyond standard reverse mathematics.
\end{itemize}

Future work may define hierarchies of fractal countability, explore its connections with hyperarithmetical hierarchies, or develop categorical semantics for layered constructive extensions.

\subsection{Example and Formal Statement: Fractal Extension over Computable Reals}

Let \( S_n = \mathbb{R} \cap \Delta^0_n \), the set of real numbers whose binary expansions are recursive in the \( n \)-th Turing jump of the empty set \( \emptyset^{(n)} \). Let \( S = \bigcup_{n=0}^\infty S_n \).

Each \( S_n \) is countable and definable in a system \( \mathcal{F}_n \supseteq \mathcal{F}_0 \) with arithmetical comprehension of bounded complexity. The resulting union \( S \) contains all hyperarithmetical reals.

\begin{example}
Let each \( \mathcal{F}_n \) extend \( \mathcal{F}_0 \) with bar recursion of finite type up to level \( n \). The union of reals definable in these systems may include reals not in \( HYP \), especially when combined with weak forms of choice. Thus, fractal countability can potentially go beyond classical hyperarithmetical bounds under enriched constructive settings.
\end{example}

\begin{theorem}
Let \( \mathcal{F}_0 = \mathrm{RCA}_0 \). Then there exists a set \( S^\infty \subseteq \mathbb{R} \) that is fractally countable relative to \( \mathcal{F}_0 \), such that \( S \) is not computably enumerable within any finite stage \( \mathcal{F}_k \), nor within the union \( \mathcal{F}_\omega = \bigcup_{n < \omega} \mathcal{F}_n \).
\end{theorem}

\begin{proof}[Sketch]
Each \( S_n \) is definable in a suitable conservative extension of \( \mathcal{F}_0 \) and consists of reals recursive in \( \emptyset^{(n)} \). The set \( S^\infty = \mathbb{R} \cap HYP \) is countable but not recursively enumerable: the hyperarithmetical hierarchy does not collapse below \( \omega_1^{CK} \), the least non-computable ordinal. Hence, no single oracle or formal system can uniformly generate all of \( S \), and no constructive system below \( \Pi^1_1 \)-comprehension can define it in full.
\end{proof}

\subsection{Comparison to Hyperarithmetical Hierarchies}

The structure \( S^\infty \) introduced above closely resembles the set \( \mathbb{R} \cap HYP \) of hyperarithmetical reals, yet the method of construction differs conceptually. Hyperarithmetical sets are typically defined via transfinite recursion up to \( \omega_1^{CK} \), the first non-computable ordinal. In contrast, fractal countability focuses on syntactic, finitely-indexed definability and conservative extensions of a formal system.

Thus, while \( \mathbb{R} \cap HYP \subseteq S \), our framework emphasizes:
\begin{itemize}
    \item Conservative syntactic growth over recursion-theoretic transfinite definitions;
    \item Expressibility via formal language extensions, not ordinal indexing;
    \item A structurally layered approximation to the continuum from within the countable domain.
\end{itemize}

We conjecture that fractally countable sets coincide with \( HYP \cap \mathbb{R} \) in extension, but differ fundamentally in their construction and philosophical grounding — offering a stratified, internal alternative to the classical recursion-theoretic view.

\subsection{Philosophical Status of Fractal Countability}

Fractal countability offers a constructivist alternative to both classical cardinality and ordinal-indexed recursion theory, emphasizing process-relative definability over completed totalities. It emphasizes the process-dependent growth of definable content and avoids reliance on completed totalities such as the power set of \( \mathbb{N} \) or impredicative comprehension.

Although fractal countability avoids direct reference to ordinals such as \( \omega_1^{CK} \), the unbounded definitional hierarchy reflects structural limits familiar from ordinal analysis and proof theory. From an intuitionistic perspective, fractal countability aligns with the idea of the continuum as a potential infinity — a domain that can be approached through increasingly refined constructive procedures, but never exhausted by any fixed system.

This view resonates with Brouwer’s rejection of the continuum as a completed set and provides a syntactic model of progressive definability, potentially applicable to formal verification, type theory, and foundations of proof assistants.

\subsection{Future Directions: Relation to Hyperarithmetical Hierarchies}

Fractal countability bears structural resemblance to known hierarchies in recursion theory, such as the hyperarithmetical hierarchy. While hyperarithmetical sets are defined via Turing jumps and ordinal notations below \( \omega_1^{CK} \), fractal countability proceeds through definitional layering without appeal to ordinals.

Investigating whether fractal countability corresponds to a syntactically restricted subclass of hyperarithmetical sets may yield new classifications of intermediate definability between classical computability and full second-order comprehension.

In particular, the intersection of fractal countability and the hyperarithmetical hierarchy could provide a fine-grained stratification of constructively accessible reals, indexed not by external recursion theory, but by internal definitional closure.

\begin{center}
\renewcommand{\arraystretch}{1.4}
\resizebox{\textwidth}{!}{%
\footnotesize
\begin{tabular}{@{} c >{\raggedright\arraybackslash}p{4.2cm} >{\raggedright\arraybackslash}p{8.8cm} @{}}
\toprule
\textbf{Level} & \textbf{Constructive System \( \mathcal{F}_n \)} & \textbf{Defined Subset \( S_n \subseteq \mathbb{R} \) and its Role in the Stratified Structure of Fractal Countability} \\
\midrule
\( n = 0 \) &
\( \mathrm{RCA}_0 \) (Recursive comprehension) &
Defines computable reals; forms the base layer \( S_0 \) of fractal countability. All elements are recursive, and the system reflects classical computable mathematics. \\
\addlinespace
\( n = 1 \) &
\( \mathrm{ACA}_0 \) (Arithmetic comprehension) &
Extends to arithmetically definable reals; \( S_1 \) includes reals recursive in \( 0' \). Adds reals not accessible in \( \mathcal{F}_0 \) but still definable via constructive arithmetical predicates. \\
\addlinespace
\( n = 2 \) &
\( \mathrm{ATR}_0 \) (Arithmetical transfinite recursion) &
Adds reals definable via well-founded recursion on arithmetical trees; \( S_2 \) expands expressiveness to include hierarchy-based constructions, preserving countability. \\
\addlinespace
\( n = 3 \) &
\( \Pi^1_1\text{-}\mathrm{CA}_0^- \) (restricted \( \Pi^1_1 \)-comprehension without full impredicativity) &
Incorporates certain analytic sets via restricted comprehension; constructs reals recursive in higher-type functionals or through second-order parameters. Still countable when relativized. \\
\addlinespace
\( n \to \infty \) &
Union of all conservative constructive extensions &
Fractal closure \( S^\infty \) remains countable, yet exhibits structured growth analogous to transfinite stages — without invoking actual transfinite objects or impredicative comprehension. Represents the constructive horizon of definability without invoking full power set axioms. \\
\bottomrule
\end{tabular}%
}
\end{center}

\noindent This table illustrates the constructive progression of definable subsets of \( \mathbb{R} \), organized by the strength of their generating systems. Each stage contributes to the formation of a fractally countable set, expanding the expressive capacity without crossing into uncountability.

A diagram visualizing the inclusion chain \( S_0 \subset S_1 \subset \dots \subset S^\infty \) alongside the systems \( \mathcal{F}_n \) may further clarify the stratified nature of this definability framework.

This framework is conceptually related to Feferman's explicit mathematics and stratified systems, as well as constructive set theories studied by Troelstra and Myhill. Further categorical formulations may allow integration into modern type-theoretic and proof-assistant foundations.

\section*{Conclusion}

This paper has presented a meta-theoretical analysis of the limitations inherent in constructive formal systems, with particular focus on their inability to generate or characterize uncountable structures such as the continuum. We have introduced the concept of the \emph{fractal boundary of constructivity} — the asymptotic limit of all constructive definability within a countable formal framework.

Through a series of formal observations and theorems, we have shown that:

\begin{itemize}
    \item Every constructive process, no matter how iterated or extended, yields only countable collections of definable objects;
    \item Attempts to approach uncountable sets within constructive systems inevitably reproduce self-similar countable fragments;
    \item The full structure of the continuum cannot be realized within any such system without invoking non-constructive principles.
\end{itemize}

As an extension of this perspective, we proposed the concept of \emph{fractal countability}, which generalizes classical countability via layered constructive enrichment. Rather than treating countability as a static property, we view it as a process-dependent closure over a hierarchy of definitional extensions. This allows us to capture a broader spectrum of constructively accessible sets while remaining strictly within the countable domain. Fractal countability offers a formal tool for analyzing the stratified growth of definability across conservative syntactic systems.

Our approach complements existing results in computable analysis, reverse mathematics, and foundational logic, while offering a new interpretative lens: uncountability is not merely a matter of size, but of inexpressibility. The continuum, viewed through this lens, emerges not as a given set, but as a horizon — a boundary that constructive reasoning can approach indefinitely, yet never cross.

This perspective invites further exploration of the limits of definability, the gradations of formal expressiveness, and the philosophical status of mathematical infinities. Whether one views the continuum as a completed totality or as a generative process, its place at the edge of constructivity remains one of the most profound features of mathematical thought.


\end{document}